\documentclass[A4paper,reqno,oneside,12pt]{amsart}
\usepackage{fullpage}
\usepackage{setspace}
\usepackage{amsmath, amssymb, amsfonts, amsthm, bbm, mathtools, mathrsfs}
\usepackage{tikz}
\usepackage{subcaption}
\usepackage{verbatim,enumitem,graphics}
\usepackage{xcolor}
\usepackage[backref=page]{hyperref}
\hypersetup{%
    colorlinks,
    linkcolor={red!50!black},
    citecolor={green!50!black},
    urlcolor={blue!50!black}
}
\usepackage{dsfont}
\usepackage{lineno}
\usepackage[normalem]{ulem}

\allowdisplaybreaks

\theoremstyle{plain}
\newtheorem{theorem}{Theorem}[section]
\newtheorem*{theorem*}{Theorem}
\newtheorem{lemma}[theorem]{Lemma}

\newtheorem{proposition}[theorem]{Proposition}

\newtheorem{corollary}[theorem]{Corollary}

\newtheorem{problem}[theorem]{Problem}

\theoremstyle{definition}

\newtheorem*{remark*}{Remark}

\newcommand{\scC}{\mathscr{C}}

\newcommand{\scL}{\mathscr{L}}

\newcommand{\scT}{\mathscr{T}}
\newcommand{\scU}{\mathscr{U}}

\newcommand{\bbE}{\mathbb{E}}

\newcommand{\bbP}{\mathbb{P}}

\newcommand{\F}{\mathbb{F}}

\def\a{\alpha}
\def\d{\delta}

\def\l{\ell}

\newcommand{\maxcut}{\operatorname{maxcut}}

\newcommand{\kissat}{\texttt{kissat}}
\newcommand{\glucose}{\texttt{glucose}}

\begin{document}

\setstretch{1.27}

\title{Some remarks on Folkman graphs for triangles}

\author{Eion Mulrenin}
\address{Department of Mathematics, Emory University, Atlanta, GA, 30322, USA}
\email{eion.mulrenin@emory.edu}

\author{Steven van Overberghe}
\address{Department of Mathematics, Computer Science, and Statistics, Ghent University, 9000 Ghent, Belgium}
\email{steven.vanoverberghe@ugent.be}

\thanks{{\it Keywords}. Ramsey theory, Folkman's theorem, finite geometry.}


\begin{abstract}
    Folkman's theorem asserts the existence of graphs $G$ which are $K_4$-free, but which have the property that every two-coloring of $E(G)$ contains a monochromatic triangle.
    The quantitative aspects of $f(2,3,4)$, the least $n$ such that there exists an $n$-vertex graph with both properties above, are notoriously difficult;
    a series of improvements over the span of two decades witnessed the solution to two \$100 Erd\H{o}s problems, and the current record due to Lange, Radziszowski, and Xu now stands at $f(2,3,4) \leq 786$,
    with another \$100 problem of Graham asking for a proof that $f(2,3,4) < 100$.

    In this paper, we study Folkman-like properties of a sequence $H_q$ of finite geometric graphs constructed using Hermitian unitals in projective planes and present some evidence that the graph $H_3$, which has 63 vertices, might contain a Folkman graph as a proper subgraph. 
    More precisely, we first prove that for all prime powers $q \geq 3$, there exists a system $\mathscr{T}_q$ of triangles in $H_q$ such that no four span a $K_4$ in $H_q$, but every two-coloring of $E(H_q)$ induces a monochromatic triangle in $\mathscr{T}_q$.
    We then show that a certain random alteration of $H_q$ which destroys all of its $K_4$'s will, for large $q$, maintain the Ramsey property with high probability.
\end{abstract}

\maketitle


\section{Introduction}
\label{section: intro}

Resolving a problem of Erd\H{o}s and Hajnal, Folkman~\cite{F70} constructed, for all $s \geq 3$, graphs $G$ which are $K_{s+1}$-free, but with the property that every two-coloring of $E(G)$ contains a monochromatic copy of $K_s$.
The latter phenomenon is customarily denoted by $G \longrightarrow K_s$, and we will use this notation throughout the paper.
HFolkman's theorem was then extended several years later by Ne\v{s}et\v{r}il and R\"odl~\cite{NR77} to edge-colorings with an arbitrary number of colors.
For $r$ and $s < t$ positive integers, graphs which are $K_t$-free but in which every $r$-edge-coloring induces a monochromatic $K_s$ are usually called {\it $(r, s, t)$-Folkman graphs}, and the least $n$ for which there exists an $(r,s,t)$-Folkman graph on $n$ vertices is customarily denoted $f(r,s,t)$.

Folkman numbers $f(r,s,t)$ are notoriously difficult.
Indeed, not only is the value of $f(2,3,4)$---the smallest nontrivial combination of parameters one could fix---unknown, but for some time, the best upper bounds were enormous (greater than a ten-times iterated exponential).
Erd\H{o}s initially offered \$100 for a proof that $f(2,3,4) \leq 10^{10}$, and this was resolved by Spencer~\cite{S88}, who optimized a random graph approach by Frankl and R\"odl~\cite{FR86}.
Erd\H{o}s then offered another \$100 for a proof that $f(2,3,4) \leq 10,000$, which was 
shown
independently by Lu~\cite{L07} and by Dudek and R\"odl~\cite{DR08-1, DR08-2} via computer-assisted proofs.
The current best bounds, due to Bikov and Nenov~\cite{BN20} and Lange, Radziszowski, and Xu~\cite{LRX14}, are still quite far apart:
\begin{equation*}
    21 \leq f(2,3,4) \leq 786,
\end{equation*}
and 
Graham has set yet another \$100 prize for a proof that $f(2,3,4) \leq 100$ 
(see, e.g.,~\cite[Problem 1.48]{G-AIM}).

More generally, the constructions of Folkman and of Ne\v{s}et\v{r}il and R\"odl are both iterative in nature, and so the upper bound on $f(r,s,t)$ was, for quite some time, utterly enormous as a function of $r$, $s$ and $t$.
Several years ago, however, R\"odl, Ruci\'nski, and Schacht~\cite{RRS17} used random graphs and hypergraph containers to show the (non-constructive) exponential-type bound $f(r,s,s+1) \leq \exp{c(s^4 \log s + s^3 r \log r)}$ for all $s \geq 3$ and $r \geq 2$, where $c > 0$ is a (large) absolute constant.
This was subsequently improved to the current state of the art
\begin{equation}
\label{eq: f(r,s,s+1)}
    f(r,s,s+1) \leq \exp(c rs^3 \log r)
\end{equation}
by Balogh and Samotij~\cite{BS20}, who refined the approach of R\"odl et al.~with an improved container lemma.


In this paper, we study Folkman-like properties of a sequence of finite geometric graphs $H_q$, defined for the parameter $q$ a prime power---see Section~\ref{section: unital} for the details of the construction.
Here and throughout, for a graph $H$ we will denote by $\binom{H}{K_3}$ the family of all triangles in $H$, and for $\scT \subseteq \binom{H}{K_3}$, we will write:
\begin{itemize}
    \item $\scT \nsupseteq K_4$ to signify that no four triangles in $\scT$ span a $K_4$ in $H$;
    \item $H \longrightarrow (K_3)_{\scT}$ to signify that every two coloring of $E(H)$ yields a monochromatic triangle in $\scT$.
\end{itemize}
Note that every Folkman graph possesses both of these properties---just take $\scT$ to be the system of all triangles---but that a graph which satisfies both is not necessarily a Folkman graph, since it might contain ``accidental" $K_4$'s which include at least one triangle that is not in $\scT$.

In this notation, our first result is the following.

\begin{theorem}
\label{thm: main-1}
    For all prime powers $q \geq 3$, there exists graph $H_q$ on $q^4-q^3+q^2$ vertices  with a system $\scT_q \subseteq \binom{H_q}{K_3}$ of triangles such that 
    \begin{equation}
    \label{eq: quasi-folkman}
        \scT_q \nsupseteq K_4 \hspace{2cm} \text{and} \hspace{2cm} H_q \longrightarrow (K_3)_{\scT_q}.
    \end{equation}
    Moreover, as $q \to \infty$, every two-coloring of $E(H_q)$ induces at least $(\frac{1}{4}-o(1))|\scT_q|$ monochromatic triangles in $\scT_q$.
\end{theorem}

Note that $H_3$ is thus a graph on $63$ vertices satisfying~\eqref{eq: quasi-folkman}.
The proof of Theorem~\ref{thm: main-1} for all $q \geq 4$ is computer-free, while the proof for $q=3$ is computer-assisted and we include its details in the appendix.
We would also like to remark that the fraction $1/4$ is best possible, as witnessed by a uniform random coloring.

Our next result shows that for large $q$, $H_q$ contains (true) Folkman graphs as subgraphs.

\begin{theorem}
\label{thm: main-2}
    There exists a prime power $q_0$ such that the following holds.
    For all prime powers $q \geq q_0$, one may delete a subset of edges of $H_q$ to make it $K_4$-free, but preserve the property that any two-coloring of the remaining edges yields a monochromatic $K_3$. 
\end{theorem}

\noindent The bound on $f(2,3,4)$ given by Theorem~\ref{thm: main-2} is quite poor, roughly $2^{280}$.
However, in light of these results, it seems plausible that $H_3$ contains a genuine Folkman graph as a proper subgraph.\\

\noindent {\bf Organization.}
The remainder of the paper is organized as follows: 
in Section~\ref{section: unital}, we define the graphs $H_q$ and highlight the properties that we will need;
in Section~\ref{section: proof of Thm 1}, we prove Theorem~\ref{thm: main-1} for $q \geq 4$, and in Section~\ref{section: proof of Thm 2}, we prove Theorem~\ref{thm: main-2}; in Section~\ref{section: conclusion}, we offer concluding remarks and discuss some open problems; and finally, in the Appendix, we include the details of our computational proof of the $q=3$ case of Theorem~\ref{thm: main-1}.


\section{Preliminaries on the Hermitian unital}
\label{section: unital}

\subsection{The intersection graph $H_q$}
Fix a prime power $q$, and let $PG(2,q^2)$ be the projective plane of order $q^2$.
A {\it classical} or {\it Hermitian} unital $\scU_q$ in $PG(2,q^2)$ is a set of points corresponding to the following set of one-dimensional subspaces of $\F_{q^2}^3$:
\begin{equation*}
    \scU_q = \{ \langle X, Y, Z \rangle: X^{q+1} + Y^{q+1} + Z^{q+1} = 0 \}.
\end{equation*}
Here, as usual, $\langle X, Y, Z \rangle$ denotes the one-dimensional subspace generated by the nonzero vector $(X,Y,Z) \in \F_{q^2}^3$.

A simple counting argument using the norm function shows that $|\scU_q| = q^3+1$.
Moreover, it can be shown that every line of $PG(2,q^2)$ intersects $\scU_q$ in either exactly one point or exactly $q+1$ points---see
Barwick and Ebert~\cite[\S 2.1]{BE} 
for proofs of these facts.
Lines of the former type are called {\it tangents}, while lines of the latter type are called {\it secants}.
Hereafter, we will denote the set of secants of $\scU_q$ by $\scL_q$.

We will be interested in the partial linear space $(\scU_q, \scL_q)$ composed of the points of the unital and the secant lines.
Perhaps one of the most intriguing features of $(\scU_q, \scL_q)$ is that it does not contain four lines in general position, somtimes referred to as the {\it O'Nan configuration}~\cite{O72}; 
see Figure~\ref{fig: o'nan}.

\begin{figure}
    \centering
    \begin{tikzpicture}

        \coordinate (A) at (-1,-1.5);
        \coordinate (B) at (-2,-1);
        \coordinate (C) at (-1,2.5);
        \coordinate (D) at (1,4.5);
        \coordinate (E) at (3,4.5);
        \coordinate (F) at (5,2.5);
        \coordinate (G) at (6,-1);
        \coordinate (H) at (5,-1.5);

        \draw (A) -- (E);        
        \draw (B) -- (F);        
        \draw (C) -- (G);        
        \draw (D) -- (H);        

        \foreach \point in {
            (0,0), (1,1.5), (2,3), (3,1.5), (4,0), (2,1)}
            {\filldraw[black] \point circle (2pt);
            }
    \end{tikzpicture}
    \caption{The four lines and six points form an O'Nan configuration.}
    \label{fig: o'nan}
\end{figure}
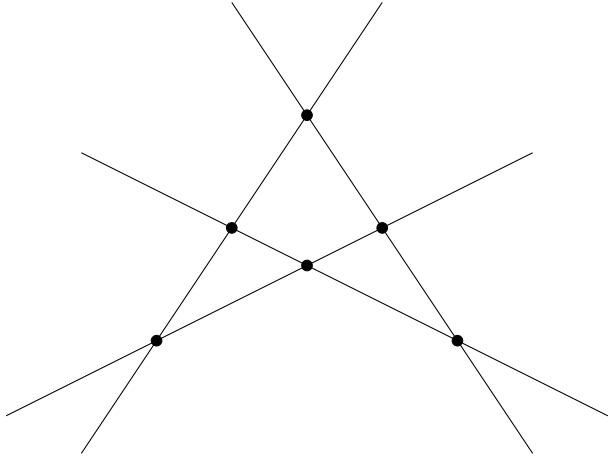

The Hermitian unital has been used recently to make progress on some longstanding problems, including the improved lower bound on the off-diagonal Ramsey number $r(4,t)$ by Mattheus and Verstra\"ete~\cite{MV24} and work by several authors on the Erd\H{o}s--Rogers problem~\cite{JS25, MV25}.
Here is a collection of basic facts we will need about the partial linear space $(\scU_q, \scL_q)$, proven originally by O'Nan~\cite{O72} and more recently by Mattheus and Verstra\"ete~\cite{MV24} using elementary counting arguments.

\begin{proposition}
\label{prp: unital}
    Let $(\scU_q, \scL_q)$ be the partial linear space formed by a Hermitian unital and its secant lines over $PG(2,q^2)$.
    We have the following.
    \begin{itemize}
        \item[i.] $|\scU_q| = q^3 + 1$ and $|\scL_q| = q^4 - q^3 + q^2$.
        \item[ii.] Every point of $\scU_q$ has $q^2$ secants passing through it and one tangent, and every secant contains $q+1$ points from $\scU_q$.
    \end{itemize}
    As a consequence,
    \begin{itemize}
        \item[iii.] For every pair of secants which intersect in $\scU_q$, there are $2q^2-2$ other secants which intersect them both in $\scU_q$.
        \item[iv.] For every pair of secants which intersect outside of $\scU_q$, there are $(q+1)^2$ other secants which intersect them both in $\scU_q$.
    \end{itemize}
\end{proposition}

\noindent Let $H_q$ denote the intersection graph of $(\scU_q, \scL_q)$, i.e., the vertex set of $H_q$ is $\{v_\l: \l \in \scL_q\}$ and $v_{\l_1} \sim v_{\l_2}$ in $H_q$ if and only if $\l_1 \cap \l_2 \in \scU_q$.
The data in Proposition~\ref{prp: unital} and the absence of O'Nan configurations in $(\scU_q, \scL_q)$ immediately translate into the following graph theoretic statistics about $H_q$.

\begin{proposition}
\label{prp: Hq}
    The graph $H_q$ is an $n$-vertex $d$-regular graph with the following properties.
    \begin{itemize}
        \item[1.] $n = q^4-q^3+q^2$ and $d = (q+1)(q^2-1) = q^3 + q^2 - q -1$.
        \item[2.] There is a set $\scC_q$ of $q^3+1$ (maximal) cliques of order $q^2$, every two of which share exactly one vertex.
        \item[3.] Each vertex is in exactly $q+1$ cliques from $\scC_q$, and every edge in exactly one clique from $\scC_q$.
        \item[4.] Every copy of $K_4$ in $H_q$ has at least three of its vertices in a clique from $\scC_q$.
    \end{itemize}
    Furthermore, $H_q$ is strongly regular with the following: 
    \begin{itemize}
        \item[5.] adjacent vertices have $2q^2 - 2$ common neighbors.
        \item[6.] non-adjacent vertices have $(q+1)^2 = q^2 + 2q + 1$ common neighbors.
    \end{itemize}
\end{proposition}

\subsection{The random block construction and non-degenerate triangles}
\label{section: random-block-constr}

The so-called {\it random block construction} takes a graph like $H_q$ as its input and randomly replaces each maximal clique by a triangle-free graph, typically just a complete bipartite graph.
The idea is that, by Proposition~\ref{prp: Hq}(4), the resulting graph is (deterministically) $K_4$-free, but might still possess interesting properties (e.g., few large independent sets~\cite{MV24}, triangles in all large subsets~\cite{JS25, MV25}, etc.) with positive probability.
This idea was introduced by Brown and R\"odl~\cite{BR03}, and has been applied by a number of authors to much success---see~\cite{MV24} and the references therein for a nice history of this method.


By design, when we apply the random block construction to $H_q$, the triangles wholly contained in a fixed clique of $\scC_q$, which come from three lines all incident to a common point, will be destroyed and the graph will become $K_4$-free. 
We will call such triangles {\it degenerate}, while triangles which are not entirely contained in a clique from $\scC_q$ (which come from three lines intersecting in three distinct points, and thus potentially survive the deletion process of the random block construction) we will call {\it non-degenerate}---see Figure~\ref{fig: triangles} for an illustration.
The set of non-degenerate triangles in $H_q$ will be our distinguished subset $\scT_q$.
Note that each vertex $v_\l \in V(H_q)$ is contained in $(q^3-q)\binom{q+1}{2}$ non-degenerate triangles, as there are $\binom{q+1}{2}$ pairs of secants which pass through both $\l$ and through each one of the $q^3-q$ points of $\scU_q \setminus \l$;
thus,
\begin{equation}
\label{eqn: non-degenerate triangles}
    |\scT_q| = \frac{1}{3}|V(H_q)|(q^3-q)\binom{q+1}{2} = \frac{1}{6} (q^4-q^3+q^2)(q^3-q)(q+1)q
\end{equation}

\begin{figure}
    \centering
    \begin{subfigure}[b]{.45\textwidth}  
        \begin{tikzpicture}
            \coordinate (A) at (-1,-1);
            \coordinate (B) at (-2,0);
            \coordinate (C) at (0,2);
            \coordinate (D) at (2,2);
            \coordinate (E) at (4,0);
            \coordinate (F) at (3,-1);

            \draw (A) -- (D) node[at end, above] {$\ell_3$};        
            \draw (B) -- (E) node[near start, above] {$\ell_1$};        
            \draw (C) -- (F) node[at start, above] {$\ell_2$};        

            \foreach \point in {
                (0,0), (1,1), (2,0)}
                {\filldraw[black] \point circle (2pt);
                }
        \end{tikzpicture}
        \centering
        \caption{The vertices $v_{\l_1}, v_{\l_2}$, and $v_{\l_3}$ will form a non-degenerate triangle in $H_q$.}
        \label{fig: non-degenerate triangles}
    \end{subfigure}
    \hfill
    \begin{subfigure}[b]{.45\textwidth}  
        \begin{tikzpicture}
            \coordinate (A) at (-0.5,-1);
            \coordinate (B) at (-3,0);
            \coordinate (C) at (-1,2);
            \coordinate (D) at (1,2);
            \coordinate (E) at (3,0);
            \coordinate (F) at (0.5,-1);

            \draw (A) -- (D) node[at end, above] {$\ell_3$};        
            \draw (B) -- (E) node[near start, above] {$\ell_1$};        
            \draw (C) -- (F) node[at start, above] {$\ell_2$};        

            \foreach \point in {
                (0,0)}
                {\filldraw[black] \point circle (2pt);
            }
        \end{tikzpicture}
        \centering
        \caption{The vertices $v_{\l_1}, v_{\l_2}$, and $v_{\l_3}$ will form a degenerate triangle in $H_q$.}
        \label{fig: degenerate triangles}
    \end{subfigure}
    \caption{ }
    \label{fig: triangles}
\end{figure}

We make note of the following proposition, which follows immediately from the fact that if the four triangles in a $K_4$ on vertices $v_{\l_1}, v_{\l_2}, v_{\l_3}, v_{\l_4} \in V(H_q)$ were all non-degenerate, then the secants $\l_1, \l_2, \l_3, \l_4$ would form an O'Nan configuration in $(\scU_q, \scL_q)$.

\begin{proposition}
\label{prp: non-degenerate triangles}
    No four triangles in $\scT_q$ induce a $K_4$ in $H_q$.
\end{proposition}

\subsection{Cliques in $H_q$}

The aim of this short subsection is to prove the following lemma, which roughly states that the clique structure of $H_q$ is inherited by the neighborhood of any particular vertex.

\begin{lemma}
\label{lem: nbhd}
    Let $v_\l \in V(H_q)$ be the vertex in $H_q$ corresponding to a secant $\l \in \scL_q$, and let $H_q[N(v_\l)]$ be the subgraph of $H_q$ induced on $N(v_\l)$.
    Then $H_q[N(v_\l)]$ is composed of:
    \begin{itemize}
        \item $q+1$ cliques of order $q^2-1$ which, with $v_\l$, form the maximal cliques in $\scC_q$ containing $v_\l$;
        \item $q^3-q$ cliques of order $q+1$ which lie inside the cliques of $\scC_q$ not containing $v_\l$.
    \end{itemize}
    Moreover, every edge in $H_q[N(v_\l)]$ is contained in exactly one such clique.
\end{lemma}

\begin{proof}
    Consider the points $p_1, \dots, p_{q+1} \in \scU_q$ which lie on $\l$.
    For each one of these points, there exists a maximal clique in $\scC_q$ containing $v_\l$ and $q^2-1$ other vertices---these come from the $q^2-1$ other secants passing through the point---and each of these cliques are pairwise vertex-disjoint aside from their common vertex $v_\l$.

    On the other hand, for every point $p \in \scU_q \setminus \l$, there exactly $q+1$ lines in $PG(2,q^2)$ which pass through both $p$ and one of the $p_i$'s; but since these lines contain two points of $\scU_q$, they are all necessarily secants, and so for each of these $q^3-q$ points, there exists a clique of $q+1$ vertices in $H_q[N(v_\l)]$.

    Finally, since all of the cliques we consider come from distinct points in $\scU_q$, i.e., lie inside cliques of $\scC_q$, no two can intersect in more than one vertex, and since we have considered every point of $\scU_q$, every edge we consider necessarily emerges from two lines intersecting in one of the points.
\end{proof}

In what remains, we will refer to the latter set of cliques as {\it spanning}, and for each $v \in V(H_q)$, we will primarily be interested in the edges of $H_q[N(v)]$ which come from spanning cliques therein.
Indeed, the endpoints of each such edges are precisely the pairs which form a non-degenerate triangle with $v$---see Figure~\ref{fig: non-degenerate triangles}.
With this in mind, for each $v_\l \in V(H_q)$, let $G_{v_\l}$ be the graph on vertex set $N(v_\l)$ with $v_{\l_1} \sim v_{\l_2}$ in $G_{v_\l}$ if and only if $\l_1 \cap \l_2 \in \scU_q \setminus \l$.
By Lemma~\ref{lem: nbhd}, $G_{v_\l}$ is the union of $q^3-q$ edge-disjoint spanning cliques of order $q+1$.

\subsection{Notation}
\label{subsection: notation}
Here we collect and define some notation based on the discussion of this section which we will stick to throughout the rest of the paper.
\begin{itemize}
    \item $H_q$ is the intersection graph of a Hermitian unital $\scU_q$ over $PG(2,q^2)$ and its secant lines $\scL_q$, with vertex set $\{v_\l: \l \in \scL_q\}$ and edges $\{v_{\l_1}, v_{\l_2}\}$ for all $\l_1, \l_2 \in \scL_q$ intersecting in $\scU_q$.
    \item $\scC_q$ is the set of $q^3+1$ maximal cliques of order $q^2$ in $H_q$ which come from taking all secants passing through a given point.
    \item $\scT_q$ is the set of {\it non-degenerate} triangles in $H_q$, i.e., triangles which come from three lines in general position. More formally, $\scT_q$ collects all triangles on vertices $v_{\l_1}, v_{\l_2}, v_{\l_3} \in \scL_q$ which have $\l_1 \cap \l_2, \l_2 \cap \l_3, \l_1 \cap \l_3 \in \scU_q$ all distinct.
    See Figure~\ref{fig: non-degenerate triangles}.
    \item For each vertex $v \in V(H_q)$, we let $G_v$ denote the graph with vertex set $N(v)$ and edges which lie in spanning cliques in $N(v)$, i.e., those edges between pairs of vertices which form a non-degenerate triangle with $v$ in $H_q$. 
\end{itemize}


\section{Proof of Theorem~\ref{thm: main-1}}
\label{section: proof of Thm 1}

Keeping in line with the notation defined in Section~\ref{subsection: notation}, our distinguished family $\scT_q$ will be the set of non-degenerate triangles in $H_q$ coming from three secants in general position.
As noted in Proposition~\ref{prp: non-degenerate triangles} above, it is clear that no four non-degenerate triangles in $H_q$ can span a $K_4$, as otherwise their corresponding secants would form an O'Nan configuration in $(\scU_q, \scL_q)$.

The heart of the proof of Theorem~\ref{thm: main-1} lies in a modification of Goodman's formula~\cite{G59} 
(cf. \cite[Lemma 2]{Con})
for the number of monochromatic triangles in a two-coloring of $E(K_n)$.
The Goodman formula has been used often 
\cite{DR08-1, DR08-2, FR86, LRX14, S88}
in the study of $f(2,3,4)$.
We modify it to count the number of monochromatic triangles from a specific family $\scT$.

\begin{lemma}
\label{lem: goodman}
    Let $G = (V,E)$ be a graph, and let $\scT$ be a family of triangles in $G$.
    For a given coloring $\Delta: E \to \{R,B\}$ and a vertex $v \in V$, let $R_\Delta(v)$ ($B_\Delta(v)$) denote the number of triangles of $\scT$ containing $v$ in which both edges incident to $v$ are red (blue).
    Then the number of triangles in $\scT$ which are monochromatic under $\Delta$ is given by
    \begin{equation}
    \label{eq: Goodman-general}
        \frac{1}{2} \left( \sum_{v \in V} {\Big (} R_\Delta(v) + B_\Delta(v) {\Big )} - |\scT| \right).
    \end{equation}
\end{lemma}

\begin{proof}
    It will be convenient to use the notation $m(\scT)$ and $nm(\scT)$ to denote, respectively, the number of monochromatic triangles in $\scT$ and the number of non-monochromatic triangles in $\scT$ under $\Delta$.
    With this, we have
    \begin{equation*}
        \sum_{v \in V} R_\Delta(v) + B_\Delta(v)
        = 3m(\scT) + nm(\scT).
    \end{equation*}
    Since $m(\scT) + nm(\scT) = |\scT|$, it follows that
    \begin{equation*}
        2 m(\scT)
        = \sum_{v \in V} {\Big (}R_\Delta(v) + B_\Delta(v) {\Big )} - |\scT|,
    \end{equation*}
    which gives the promised formula.
\end{proof}

By~\eqref{eqn: non-degenerate triangles}, we immediately get a formula which counts monochromatic triangles in $\scT_q$ under edge-colorings of $H_q$.

\begin{corollary}
\label{cor: goodman}
    Let $\Delta: E(H_q) \to \{R,B\}$ be a two-coloring, and for a vertex $v \in V(H_q)$, let $R_\Delta(v)$ ($B_\Delta(v)$) denote, as above, the number of triangles in $\scT_q$ in which both edges incident to $v$ are red (blue).
    Then the number of triangles in $\scT_q$ which are monochromatic under $\Delta$ is given by
    \begin{equation}
    \label{eq: Goodman-Hq}
        \frac{1}{2} \left( \sum_{v \in V(H_q)} {\Big (} R_\Delta(v) + B_\Delta(v) {\Big )} - \frac{1}{6} (q^4-q^3+q^2)(q^3-q)(q+1)q \right).
    \end{equation}
\end{corollary}

\begin{proof}[Proof of Theorem~\ref{thm: main-1}]
Let $\Delta: E(H_q) \to \{R,B\}$ be a given two-coloring of the edges of $H_q$, let $v \in V(H_q)$ be a fixed vertex, and let $G_v$ be the graph on $N(v)$ with edges forming non-degenerate triangles with $v$ (see Section~\ref{subsection: notation}).
We show that there necessarily exists a monochromatic triangle in $\scT_q$, which, by Proposition~\ref{prp: non-degenerate triangles}, suffices to prove the theorem.

Define an auxiliary vertex-coloring $\chi: V(G_v) \to \{R,B\}$ where
\begin{equation*}
    \chi(w) = \Delta(\{v,w\}),
\end{equation*}
i.e., $\chi$ assigns to vertices in the neighborhood of $v$ the color of their edge with $v$ under $\Delta$.
By construction, the number of monochromatic edges under $\chi$ is equal to $R_\Delta(v) + B_\Delta(v)$.
Recall that, by Lemma~\ref{lem: nbhd}, $G_v$ is composed of $q^3-q$ edge-disjoint spanning cliques of order $q+1$.
For a fixed such clique, the number of monochromatic edges it contains is 
\begin{equation}
\label{eq: mono-edges}
    \binom{|\chi^{-1}(R_\Delta(v))|}{2} + \binom{|\chi^{-1}(B_\Delta(v))|}{2}.
\end{equation}
Since $|\chi^{-1}(R_\Delta(v))| + |\chi^{-1}(B_\Delta(v))| = q+1$, the expression \eqref{eq: mono-edges} is minimized when $|\chi^{-1}(R_\Delta(v))| = |\chi^{-1}(B_\Delta(v))| = (q+1)/2$ by the convexity of the function $\binom{x}{2}$.
Therefore, each clique contributes at least $2 \binom{\frac{q+1}{2}}{2}$ monochromatic edges under $\chi$, and since these cliques are edge-disjoint, the total number of monochromatic edges in $G_v$ (under $\chi$) is at least
\begin{equation*}
    (q^3-q) \cdot 2 \binom{\frac{q+1}{2}}{2} 
    = (q^3-q) \left( \frac{1}{4}(q+1)^2 - \frac{1}{2}(q+1) \right)
\end{equation*}
Since this gives a lower bound for $R_\Delta(v) + B_\Delta(v)$ (independent of $\Delta$), summing over all $v \in V(H_q)$ we have
\begin{equation*}
    \sum_{v \in V(H_q)} R_\Delta(v) + B_\Delta(v) \geq (q^4-q^3+q^2) (q^3-q) \left( \frac{1}{4}(q+1)^2 - \frac{1}{2}(q+1) \right).
\end{equation*}
Thus, \eqref{eq: Goodman-Hq} is 
positive for all $\Delta$ when
\begin{equation}
\label{eqn: q geq 4}
    \frac{1}{4}(q+1)^2 - \frac{1}{2}(q+1) > \frac{1}{6}(q+1)q;
\end{equation}
equality holds when $q=3$, and the strict inequality holds for all $q > 3$.

Finally, note that as $q \to \infty$, the above bound gives
\begin{eqnarray*}
    \frac{1}{2} \left( \sum_{v \in V(H_q)} {\Big (} R_\Delta(v) + B_\Delta(v) {\Big )} - \frac{1}{6} (q^4-q^3+q^2)(q^3-q)(q+1)q \right)
    &\geq& \frac{1}{2} \left( \frac{1}{4}q^9 - \frac{1}{6}q^9 - O(q^8) \right)\\
    &=& \frac{1}{24}q^9 - O(q^8)\\
    &=& \left( \frac{1}{4} - o(1) \right)|\scT_q|,
\end{eqnarray*}
as claimed.
\end{proof}



\section{Proof of Theorem~\ref{thm: main-2}}
\label{section: proof of Thm 2}



We begin by observing another version Goodman's formula, which can be obtained from Lemma~\ref{lem: goodman} by simply noting that for any graph $G$, $\sum_{v \in V(G)} \frac{1}{3}e(G[N(v)])$ is equal to the number of triangles in $G$.

\begin{proposition}
\label{prp: goodman}
    Let $G = (V,E)$ be a graph, and let $\Delta: E \to \{R,B\}$ be a two-coloring of its edges.
    For a vertex $v \in V$, let $R_{\Delta}(v)$ ($B_\Delta(v)$) denote the number of edges in $N(v)$ in which both endpoints form a red (blue) edge with $v$.
    Then the number of monochromatic triangles in $G$ under $\Delta$ is given by
    \begin{equation*}
        \frac{1}{2} \sum_{v \in V} {\Big (} R_\Delta(v) + B_\Delta(v) - \frac{1}{3}|E(G[N(v)])| {\Big )}.
    \end{equation*}
\end{proposition}

Recall from Section~\ref{section: random-block-constr} that to make $H_q$ $K_4$-free, it suffices to replace each maximal clique in $\scC_q$ by a triangle-free graph.
In view of Proposition~\ref{prp: goodman}, to turn $H_q$ into a Folkman graph, we want to do this in such a way that for all $v$, we still have that $R_\Delta(v) + B_\Delta(v) > \frac{1}{3}|E(G[N(v)])|$ (to ensure the Ramsey property).
It turns out that the usual choice for a triangle-free graph, namely, a complete bipartite graph, does not quite work here.
The reason is because this latter property is captured well by graphs which have small maximum cut\footnote{Recall that for a graph $G$, the maximum cut is the maximum number of edges in a bipartite subgraph of $G$, i.e., $\maxcut(G) = \max \{ |E(G[S,T])|: S \sqcup T = V(G) \}$.} relative to the number of edges.
Indeed, in the proof of Theorem~\ref{thm: main-1}, lower-bounding the number of monochromatic edges under the auxiliary coloring $\chi$ was just upper-bounding the size of the maximum cut of $G_v$.\footnote{While much of the previous work~\cite{DR08-1, DR08-2, FR86, LRX14, S88} on $f(2,3,4)$ has been phrased in the language of maximum cuts, we have elected to use auxiliary colorings, as it simplifies our proofs and directly suggests several natural approaches towards obtaining improved bounds for more general classes of Folkman numbers using the methods of this paper.}

To that end, throughout this section, let $F$ be a triangle-free graph with the following fixed parameters:
$|V(F)| = n$, $|E(F)| = m$, and $\maxcut(F) = \a m$ for some $\a < 2/3$.
For such a graph one could take, e.g., a small instance of Alon's optimally pseudorandom $K_3$-free graphs~\cite{A94, A96}---see below for a discussion of the numerics that this specific graph gives.
We remark that much work has been done on triangle-free graphs with small maximum cuts, and we refer the interested reader to the discussion in~\cite{A96}.

The following proposition follows immediately from considering the number of edges contained in each color class if we think of a two-coloring of $V(F)$ as a cut.

\begin{proposition}
    Under any coloring $\chi: V(F) \to \{R,B\}$, there are at least $(1 - \a)m$ monochromatic edges.
\end{proposition}

Recall that a {\it $t$-blowup} of a graph $G$, denoted $G[t]$, is obtained by replacing each vertex of $G$ with an independent set of size $t$ and each edge of $G$ with a complete bipartite graph between the sets which replaced its endpoints; thus, $F[t]$ has $nt$ vertices and $mt^2$ edges.
The following lemma essentially states that the proportion of monochromatic edges under a vertex-coloring of $F$ is preserved under taking blowups.

\begin{lemma}
\label{lem: blowup-coloring}
    Let $F[t]$ be a $t$-blowup of $F$ with $nt$ vertices and $mt^2$ edges.
    Under any coloring $\chi: V(F[t]) \to \{R,B\}$ of its vertex set, there are at least $(1 - \a)mt^2$ monochromatic edges.
\end{lemma}

\begin{proof}
    Let $V = V_1 \cup \dots \cup V_n$ be the vertex set of $F[t]$, where $V_i$ corresponds to the blowup of the vertex $i \in V(F)$, and let $\chi: V \to \{R,B\}$ be a coloring of the vertex set.
    For convenience, set
    \begin{equation*}
        r_i = |\{v \in V_i: \chi(v) = R\}| \text{ and } b_i = |\{v \in V_i: \chi(v)=B\}|
    \end{equation*}
    for each $i \in [n]$, noting that $r_i + b_i = t.$
    Then the number of of monochromatic edges under $\chi$ is given by
    \begin{equation*}
        \sum_{\{i,j\} \in E(F)} r_ir_j + b_ib_j 
        = \sum_{\{i,j\} \in E(F)} r_ir_j + (t-r_i)(t-r_j) 
        = \sum_{\{i,j\} \in E(F)} t^2 - (r_i+r_j)t +2r_ir_j,
    \end{equation*}
    so we are looking for the minimum of the polynomial 
    \begin{equation*}
        f(r_1, \dots, r_n) =\sum_{\{i,j\} \in E(F)} t^2 - (r_i+r_j)t +2r_ir_j 
        \hspace{1cm} \text{ over } \hspace{1cm}
        (r_1, \dots, r_n) \in [0,t]^n.
    \end{equation*}
    Note that $f$ is multilinear, i.e., $\frac{\partial^2}{\partial r_i^2} f = 0$ for each $i \in [n]$, so its minimum is only achieved on the boundary of $[0,t]^n$, and in fact on $\{0,t\}^n$.
    This means that the minimum number of monochromatic edges in a vertex-coloring of $F[t]$ will (only) be achieved by a coloring $\chi$ which colors each $V_i$ monochromatically.
    But such a $\chi$ will reduce to a coloring $\chi'$ of $V(F)$, and we know at least $(1 - \a)m$ edges of $F$ must be monochromatic under $\chi'$;
    therefore, at least $(1 - \a)mt^2$ edges of $F[t]$ must be monochromatic under $\chi$.
\end{proof}

Our plan now is to replace each maximal clique by a random blowup of $F$, which is necessarily triangle-free just because $F$ is.
More precisely, for each $w \in C \in \scC_q$, define a random variable $X_{w,C} \in [n]$ which assigns $(w,C)$ to one of the $n$ vertices of $F$ uniformly at random and independently of all the other $X_{w',C'}$'s.
After $X_{w,C}$ is chosen for each $w \in C$, we will then take edges between all vertices $u$ and $w$ such that $u$ and $v$ both lie in some $C \in \scC_q$ and $X_{u,C}$ and $X_{w,C}$ are an edge in $F$.

For each $v \in V(H_q)$, recall that $G_v$ is the graph on $N(v)$ which is the union of the $q^3-q$ spanning cliques in $N(v)$.
Hereafter, we will denote by $H_q^*$ the subgraph of $H_q$ obtained after replacing each $C \in \scC_q$ with a random blowup of $F$; by $N^*(v)$ the neighborhood of $v$ in $H_q^*$; and by $G_v^*$ the graph $H_q^*$ induced on $N^*(v)$ (note that $E(G_v^*) = E(G_v) \cap \binom{N^*(v)}{2}$).

With this, we prove the following lemma.

\begin{lemma}
\label{lem: blowup}
    For every $\d > 0$, the following event holds with high probability, i.e., with probability approaching 1 as $q \to \infty$.
    For every $v \in V(H_q)$ and every spanning clique $C$ in $G_v$,
    each of the $n$ blown-up vertices in $C$ will have\footnote{The notation $(1 \pm \d) \frac{2m(q+1)}{n^3}$ means that the number lies between $(1 - \d) \frac{2m(q+1)}{n^3}$ and $(1 + \d) \frac{2m(q+1)}{n^3}$.} $(1 \pm \d) \frac{2m(q+1)}{n^3}$ vertices in $N^*(v)$.
\end{lemma}

The proof uses McDiarmid's bounded differences inequality (see, e.g.,~\cite{M}).


\begin{theorem}[{\cite[Theorem 3.1]{M}}]
\label{thm: mcdiarmid}
    Let $\mathbf{X} = (X_1, \dots, X_k)$ be a family of independent random variables with $X_i$ taking values in a set $A_i$ for each $i$.
    Suppose that a real-valued function $f$ defined on $\prod_{i=1}^k A_i$ satisfies
    \begin{equation*}
        |f(\mathbf{x}) - f(\mathbf{x}')| \leq c_i
    \end{equation*}
    whenever the vectors $\mathbf{x}$ and $\mathbf{x}'$ differ only in the $i$th coordinate.
    Then for any $\d > 0$,
    \begin{equation*}
        \bbP{\Big [} {\big |}f(\mathbf{X}) - \bbE[f(\mathbf{X})] {\big |} \geq \d \bbE[f(\mathbf{X})] {\Big ]}
        \leq 2\exp \left( - \frac{2 \d^2 \bbE [f(\mathbf{X})]^2}{c_1^2 + \dots + c_k^2}. \right)
    \end{equation*}
\end{theorem}

\begin{proof}[Proof of Lemma~\ref{lem: blowup}]
    Fix $v \in V(H_q)$, $C_0$ a spanning clique in $G_v$, and $i \in [n]$, and let $w_1, \dots, w_{q+1}$ be the vertices of $C_0$.
    Let $V_i(C_0)$ be the blowup of vertex $i$ in $C_0$, let $C_1, \dots C_{q+1} \in \scC_q$ be the (distinct) maximal cliques with $C_j$ containing the edge $\{v,w_j\}$, and let $f$ be the function depending on the $3(q+1)$ variables
    \begin{equation*}
        \{X_{v,C_j}: j \in [q+1]\} \cup \{X_{w_j,C_j}: j \in [q+1]\} \cup \{X_{w_j, C_0}: j \in [q+1]\}
    \end{equation*}
    which computes the number of vertices of $C_0$ which end up in $V_i(C_0) \cap N^*(v)$.
    Note that a vertex $w_j \in C_0$ appears in $V_i(C_0) \cap N^*(v)$ if and only if $w_j$ survives as a neighbor of $v$ and $X_{w_j,C_0} = i$.
    Thus, if we let
    \begin{equation*}
        \mathbf{X} = (X_{v,C_1}, \dots, X_{v,C_{q+1}}, X_{w_1,C_1}. \dots, X_{w_{q+1},C_{q+1}}, X_{w_1, C_0}, \dots, X_{w_{q+1},C_0}),
    \end{equation*}
    then $\bbE[f(\mathbf{X})] = \frac{2m(q+1)}{n^3}$, as for each vertex $w_j \in C_0$, the edge $\{v,w_j\}$ survives with probability $2m/n^2$ and $X_{w_j,C_0} = i$ with probability $1/n$, and these events are independent (as the first depends only on $X_{v, C_j}$ and $X_{w_j,C_j}$).
    Also, changing the value of any one of these $3(q+1)$ variables can change the value of $f$ by at most 1, as in the most extreme case,
    since the $C_j$'s and $C_0$ are all distinct cliques,
    it can simply either add or remove a vertex from $V_i(C) \cap N^*(v)$; hence we can take all of the bounded difference constants $c_i$ to be 1.
    Therefore, by Theorem~\ref{thm: mcdiarmid},
    \begin{equation*}
        \bbP {\Big [} {\big |} f(\mathbf{X}) - \bbE[f(\mathbf{X})] {\big |} \geq \d \bbE[f(\mathbf{X})] {\Big ]} 
        \leq 2\exp \left( - \frac{2 \cdot \d^2 \cdot \frac{4m^2}{n^6} (q+1)^2}{3(q+1)} \right) 
        = 2 \exp \left( - \frac{8 \d^2 m^2}{3n^6} (q+1) \right).
    \end{equation*}
    Taking a union bound over the $(q^4-q^3+q^2)(q^3-q) \cdot n \leq nq^7$ combinations of $v$, $C_0$ and $i \in [n]$ then shows that with probability at least
    \begin{equation}
    \label{eqn: probability}
        1 - 2 \exp \left( - \frac{8 \d^2 m^2}{3n^6} \cdot q + 7\ln(nq) \right),
    \end{equation}
    $|V_i(C) \cap N^*(v)| = (1\pm\d)\frac{2m(q+1)}{n^3}$ holds for all $v \in V(H_q^*)$, spanning cliques $C$ in $G_v$, and $i \in [n]$.
\end{proof}

The proof of Theorem~\ref{thm: main-2} is now straightforward.

\begin{proof}[Proof of Theorem~\ref{thm: main-2}]
Fix $\d > 0$ sufficiently small, fix an instance $H_q^*$ in which Lemma~\ref{lem: blowup} holds, and fix a vertex $v \in V(H_q^*)$.
Note that since $F$ is triangle-free, any blowup of $F$ is also triangle-free, and thus as explained above, $H_q^*$ is $K_4$-free.
Hence, we need to check that the Ramsey property still holds in $H_q^*$, for which we use Proposition~\ref{prp: goodman}.

To that end, fix 
a vertex $v \in V(H_q^*)$, and recall that $H_q^*[N^*(v)] = G_v^*$.
For each clique $C$ in $G_v$, in $G_v^*$ there now exists $m \cdot \left((1\pm\d)\frac{2m(q+1)}{n^3}\right)^2$ edges, and so in total, the number of edges in $G_v^*$ is at most $m \cdot (q^3-q) \cdot \left((1+\d)\frac{2m(q+1)}{n^3}\right)^2$.
However, by Lemma~\ref{lem: blowup-coloring}, it follows that for any choice of an edge-coloring $\Delta$ of $H_q^*$, each $C$ will contribute at least 
$(1 - \a) m \cdot \left((1-\d)\frac{2m(q+1)}{n^3}\right)^2$
to the sum $\sum_{v \in V(H_q^*)} R_\Delta(v) + B_\Delta(v)$, 
and so
we will have
\begin{equation*}
    \sum_{v \in V(H_q^*)} R_\Delta(v) + B_\Delta(v) \geq |V(H_q^*)|(1-\a) m \cdot (q^3-q) \cdot \left((1-\d)\frac{2m(q+1)}{n^3}\right)^2.
\end{equation*}
Therefore, by Proposition~\ref{prp: goodman} the number of monochromatic triangles in $H_q^*$ under any coloring $\Delta$ is at least
\begin{equation*}
    \frac{1}{2}|V(H_q^*)| \left( \left(1-\a\right) m \cdot (q^3-q) \cdot \left((1-\d)\frac{2m(q+1)}{n^3}\right)^2 - \frac{1}{3} m \cdot (q^3-q) \cdot \left((1+\d)\frac{2m(q+1)}{n^3}\right)^2 \right),
\end{equation*}
which is positive for $\d$ small enough by the hypothesis that $\a < 2/3$.
\end{proof}

\begin{remark*}
    In fact, we can say a bit more. 
    Note that the expected number of triangles in $H_q^*$ is $\frac{8m^3}{n^6} |\scT_q| = \Theta(q^9)$, as each edge of a triangle in $\scT_q$ survives the random block construction independently of the other two.
    Therefore, if we take an instance of $H_q^*$ where this also holds, the last equation shows that a positive proportion of the triangles in $H_q^*$ are monochromatic under any edge-coloring.
    If $\a = 1/2+o(1)$ as in Alon's triangle-free graphs or sparse random graphs, then we in fact get a $1/4-o(1)$ portion of the remaining triangles monochromatic.
\end{remark*}

\subsection{The quantitative aspects of Theorem~\ref{thm: main-2}}
\label{section: quantitative-aspects}

We proved Theorem~\ref{thm: main-2} by replacing each maximal clique in $\scC_q$ with a random blowup of a fixed graph $F$, and the proof goes through for any choice of $F$ which is both triangle-free and has a maximum cut less than $\frac{2}{3}|E(F)|$.
The probability~\eqref{eqn: probability} which we require to be positive depends on the two fixed parameters $m = |E(F)|$ and $n = |V(F)|$, a parameter $\d$ which need only be small enough for the last equation to go through (and thus may be optimized based on $m$ and $n$), and the prime power $q$ which we simply take to be sufficiently large in terms of the others.

Our choice for $F$ is based on the following theorem due to Alon~\cite{A94, A96}.
This particular statement follows from Lemma 3.1 and Proposition 3.2 in~\cite{A96}, using his construction from~\cite{A94}.

\begin{theorem}[Alon~\cite{A94, A96}]
\label{thm: Alon}
    For every $k \not\equiv 0$ (mod $3$), there exists a triangle-free graph $G_k$ on $2^{3k}$ vertices with
    \begin{equation}
    \label{eqn: alon-maxcut}
        \maxcut(G_k) \leq \frac{1}{4}2^{3k} \left( 2^{k-1}(2^{k-1}-1) + 9 \cdot 2^k + 3 \cdot 2^{k/2} + \frac{1}{4} \right)
    \end{equation}
    and
    \begin{equation}
    \label{eqn: alon-edges}
        e(G_k) = \frac{1}{2} 2^{3k} 2^{k-1}(2^{k-1}-1).
    \end{equation}
\end{theorem}

The smallest $k \not\equiv 0$ (mod 3) for which $\maxcut(G_k) < \frac{2}{3}e(G_k)$ is $k=7$, and thus we may take for $F$
\begin{equation*}
    n = |V(F)| = 2^{21} \hspace{2cm} \text{and} \hspace{2cm} m = |E(F)| = 2^{26}(2^6-1).
\end{equation*}
With this choice of $F$, the smallest $q$ satisfying Theorem~\ref{thm: main-2} will be roughly $2^{70}$, and so the graph $H_q^*$ gives an upper bound of roughly $f(2,3,4) \leq 2^{280}$ (as $|V(H_q^*)| = q^4-q^3+q^2$).
This is, however, substantially larger than even the Folkman graphs constructed by Frankl and R\"odl~\cite{FR86} and by Spencer~\cite{S88}.


\section{Concluding remarks}
\label{section: conclusion}

In this paper, we showed that for all prime powers $q \geq 4$, the graph $H_q$ has a system $\scT_q \subseteq \binom{H_q}{K_3}$ with
\begin{equation*}
    \scT_q \nsupseteq K_4 \hspace{2cm} \text{and} \hspace{2cm} H_q \longrightarrow (K_3)_{\scT_q},
\end{equation*}
and that for large $q$, one can destroy all $K_4$'s in $H_q$ while preserving the Ramsey property $H_q \longrightarrow K_3$.
We believe it is possible that the requirement that $q$ is large to be purely a defect of our method as opposed to a fundamental property of the graphs $H_q$.
On the other hand, as detailed in Appendix~\ref{section: appendix} below, we ran a number of computer-aided experiments on $H_3$ to test whether it contained any genuine Folkman graph as a subgraph, none of which succeeded.
With this in mind, we pose the following problem.
\begin{problem}
\label{problem: true-folkman}
    Decide whether the graph $H_3$ has a $K_4$-free subgraph $H$ with $H \longrightarrow K_3$.
\end{problem}
\noindent
Since $H_3$ has 63 vertices, a solution to this problem in the affirmative would resolve Graham's \$100 problem.

It is also worth mentioning that since all previous work on effective bounds for $f(2,3,4)$ have used some version of Goodman's formula or maximum cuts, which are both highly specific to two-colorings, none have admitted an extension to colorings using three or more colors.
Thus, it would be interesting to determine whether the graphs $H_q$ exhibit similar ``quasi-Folkman" properties for more than two colors.

\begin{problem}
\label{pr: multicolor}
    Show that for $r \geq 3$ and $q = q(r)$ sufficiently large, every $r$-coloring of $E(H_q)$ contains a monochromatic $K_3$ in $\scT_q$.
\end{problem}

\subsection{Acknowledgments}
The first author would like to thank Cosmin Pohoata for his encouragement and for many helpful suggestions and conversations.

\subsection{Declaration on the use of AI}
The authors declare that no AI was used in the preparation of this manuscript, and that all mathematical ideas and written components are human-generated.


\appendix
\section{Ramsey properties of $H_3$}
\label{section: appendix}

Here we present our computational results on $H_3$ and prove the $q=3$ case of Theorem~\ref{thm: main-1}.
Determining if a graph $G$ equipped with a given system of triangles $\scT  \subseteq \binom{G}{K_3}$ has the ``quasi-Folkman property" studied in this paper, i.e., if $\scT\nsupseteq K_4$ and $G \longrightarrow (K_3)_\scT$, can be done practically by transforming the property to a \textit{satisfiability} (SAT) problem.
For two-colorings, we assign one boolean variable $x_e$ to each edge $e\in E$, indicating the `color' the edge receives.
Then, for each triangle $\{v,w,z\}\in\scT$, we make two clauses: \( x_{\{v,w\}}\lor x_{\{v,z\}} \lor x_{\{w,z\}} \) and \( \overline{x_{\{v,w\}}}\lor \overline{x_{\{v,z\}}} \lor \overline{x_{\{w,z\}}} \). These indicate that not all three edges can be of the same color.
Then the task is to find an assignment of truth values to every $x_e$ such that all clauses are `true'. This is an instance of a 3-SAT problem, for which fast solvers exist.
We used two different SAT-solvers, for comparison and to reduce the risk of programming errors: \glucose~ and \kissat.

The SAT-instance corresponding to the quasi-Folkman property of $H_3$ has 1008 variables and 6048 clauses: \kissat~ needs 2 seconds to verify that the instance cannot be satisfied, i.e. $H_3$ is quasi-Folklman, while \glucose~ requires 10 seconds.  

Because these solving times are so low, we were convinced that $H_3$ has the quasi-Folkman property with a wide margin.
We did some experiments and observed that up to three maximum cliques which do not have a vertex in common can be removed while retaining the quasi-Folkman property.
The remaining graph has 39 vertices and 1488 triangles, of which 898 are non-degenerate.

Since $H_3$ is smaller than the smallest known $K_4$-free Folkman graph, we performed a number of computational experiments which try to turn $H_3$ into a \textit{true} Folkman graph (see Problem~\ref{problem: true-folkman} above).
To this end, we tried:
\begin{enumerate}
    \item replacing every maximum clique of $H_3$ with a complete bipartite graph. That is, for each maximum clique, independently pick a random permutation of its vertices and then replace the edges by those of a predefined bipartite graph;
    \item replacing every maximum clique of $H_3$ with a random maximal-triangle-free graph, in a similar way;
    \item one by one removing random edges from $H_3$ until the graph becomes $K_4$-free;
    \item one by one greedily removing the edge that is contained in the maximum amount of $K_4$'s, until the graph becomes $K_4$-free; and
    \item randomly adding non-degenerate triangles of $H_3$ to an initially empty graph, as long as the graph stays $K_4$-free.
\end{enumerate}
Every experiment was repeated thousands of times, but this never resulted in a genuine Folkman graph. Moreover, most graphs constructed via techniques (1),(2) and (3) appeared to be `trivially' colorable in the following sense: every edge that is contained in at most one triangle can never contribute to a graph being non-colorable, so it can be removed from the graph.  This in turn can cause more edges to be in at most one triangle, so we repeat the process as long as necessary.  Also, vertices with degree lower than 8 can be removed if the graph is $K_4$-free~\cite{BN20}.  Graphs from the first three experiments were then most often simplified to the empty graph, while graphs derived from experiments (4) and (5) appear to have very few removable vertices and edges.

We also did an experiment to find quasi-Folkman graphs of small order, unrelated to $H_3$.
First we guess a small order $N$, and then make a set $\scT$ of triangles by adding random triples from $\{1\dots N \}$ as long as no 4 elements in $\scT$ span only 4 points.
After only about a hundred tries, this procedure found a quasi-Folkman graph on 12 vertices.
It turned out that there are also \textit{circulant} quasi-Folkman graphs on 11 vertices, having $|\scT| = 88$.
One example is when $|\scT|$ equals all circular rotations of $$\{1,2,3\},\{1,2,4\},\{1,2,6\},\{1,2,7\},\{1,3,5\},\{1,3,8\},\{1,4,7\},\{1,4,8\}$$

The code for the experiments described in the appendix can be found at \url{https://github.com/Steven-VO/quasiFolkman}.

\end{document}